\documentclass[11pt]{article}
\usepackage{graphicx}
\usepackage{amscd, amssymb}
\usepackage{enumerate}
\usepackage{hyperref}
\usepackage{lscape}
\newenvironment{eq}{\begin{equation}}{\end{equation}}
\newenvironment{proof}{{\bf Proof}:}{\vskip 5mm }

\newtheorem{proposition}{Proposition}[subsection]
\newtheorem{lemma}[proposition]{Lemma}
\newtheorem{definition}[proposition]{Definition}

\newtheorem{remark}[proposition]{Remark}

\newcommand{\llabel}[1]{\label{#1}}
\newcommand{\comment}[1]{}
\newcommand{\sr}{\rightarrow}

\newcommand{\nn}{{\bf N\rm}}

\newcommand{\wt}{\widetilde}

{\vskip
3mm}

\begin{document}
\parskip = 2mm
\begin{center}
{\bf\Large Subsystems and regular quotients of C-systems}

\vspace{3mm}

{\large\bf Vladimir Voevodsky}\footnote{School of Mathematics, Institute for Advanced Study,
Princeton NJ, USA. e-mail: vladimir@ias.edu}$^,$\footnote{Work on this paper was supported by NSF grant 1100938.}
\vspace {3mm} 
\end{center}

\begin{abstract}C-systems were introduced by J. Cartmell under the name ``contextual categories''. In this note we study sub-objects and quotient-objects of C-systems. In the case of the sub-objects we consider all sub-objects while in the case of the quotient-objects only {\em regular} quotients that in particular have the property that the corresponding projection morphism is surjective both on objects and on morphisms.

It is one of several short papers based on the material of the "Notes on Type Systems" by the same author. 
\end{abstract}

\subsection{Introduction}
C-systems were introduced by John Cartmell (\cite{Cartmell0}, \cite[p.237]{Cartmell1}) and studied further by Thomas Streicher (see \cite[Def. 1.2, p.47]{Streicher}).  Both authors used the name “contextual categories” for these structures.  We feel it to be important to use the word ``category''  only for constructions which are invariant under equivalences of categories. For the essentially algebraic structure with two sorts ``morphisms'' and ``objects'' and operations ``source'', ``target'', ``identity'' and ``composition'' we suggest to use the word pre-category. Since the additional structures introduced by Cartmell are not invariant under equivalences we can not say that they are structures on categories but only that they are structures on pre-categories. Correspondingly, Cartmell objects should be called ``contextual pre-categories''. We suggest to use the name C-systems instead\footnote{The distinction between categories and pre-categories becomes precise in the univalent foundations where not all collections of objects are constructed from sets. See \cite{RezkCompletion} for a detailed discussion.}.

Our first result, Proposition \ref{2014.07.06.prop1}, shows that C-systems can be defined in two equivalent ways: one, as was originally done by Cartmell, using the condition that certain squares are pull-back and another using an additional operation $f\mapsto s_f$ which is almost everywhere defined and satisfies simple algebraic conditions. 

This description is useful for the study of quotients and homomorphisms of C-systems.

To any C-system $CC$ we associate a set $\wt{Ob}(CC)$ and eight partially defined operations  on the pair of sets $(Ob(CC),\wt{Ob}(CC))$. 

In Proposition \ref{2009.10.15.prop2} we construct a bijection between C-subsystems of a given C-system $CC$ and pairs of subsets $(C,\wt{C})$ in $(Ob(CC),\wt{Ob}(CC))$ which are closed under the eight operations. This provides, through the results established in \cite{Cofamodule}, an algebraic justification for what is known as the ``structural'' or ``basic'' rules of the dependent type theory (see \cite[p.585]{Jacobs1}).  More precisely, the description of subsystems constructed in the present paper provides a justification for the subset of the ``structural'' rules that concern the behavior of the type and term judgements. 

The algebraic justification for the rules that concern the type equality and the term equality judgements is achieved in Proposition \ref{2014.07.08.prop1} where we construct a bijection between {\em regular congruence relations} on $CC$ and pairs of equivalence relations on $(Ob(CC),\wt{Ob}(CC))$ which are compatible with the eight operations and satisfy some additional properties. 

Besides their role in the mathematical theory of the syntactic structures that arise in dependent type theory these two results strongly suggest that the theory of C-systems is equivalent to the theory  with the sorts $(Ob,\wt{Ob})$ and the eight operations which we consider together with some relations among these operations. 

The essentially algebraic version of this other theory is called the theory of B-systems and will be considered in the sequel \cite{Bsystems}.

This is one of the short papers based on the material of \cite{NTS} by the same author. I would like to thank the Institute Henri Poincare in Paris and the organizers of the ``Proofs'' trimester for their hospitality during the preparation of this paper. The work on this paper was facilitated by discussions with Richard Garner and Egbert Rijke.

\subsection{C-systems}
By a pre-category $C$ we mean a pair of sets $Mor(C)$ and $Ob(C)$ with four maps
$$\partial_0,\partial_1:Mor(C)\sr Ob(C)$$
$$Id:Ob(C)\sr Mor(C)$$
and 
$$\circ:Mor(C)_{\partial_1}\times_{\partial_0} Mor(C)\sr Mor(C)$$
which satisfy the well known conditions of unity and associativity (note that we write composition of morphisms in the form $f\circ g$ or $fg$ where $f:X\sr Y$ and $g:Y\sr Z$). These objects would be usually called categories but we reserve the name ``category'' for those uses of these objects that are invariant under the equivalences.  
\begin{definition}
\llabel{2014.07.06.def3}
A C0-system is a pre-category $CC$ with additional structure of the form
\begin{enumerate}
\item a function $l:Ob(CC)\sr \nn$,
\item an object $pt$,
\item a map $ft: Ob(CC)\sr Ob(CC)$, 
\item for each $X\in Ob(CC)$ a morphism $p_X:X\sr ft(X)$,
\item for each $X\in Ob(CC)$ such that $l(X)>0$ and each morphism $f:Y\sr ft(X)$ 
an object $f^*X$ and a morphism $q(f,X):f^*X\sr X$,
\end{enumerate}
which satisfies the following conditions:
\begin{enumerate}
\item $l^{-1}(0)=\{pt\}$
\item for $X$ such that $l(X)>0$ one has $l(ft(X))=l(X)-1$
\item $ft(pt)=pt$
\item $pt$ is a final object,
\item for $X\in Ob(CC)$ such that $l(X)>0$ and $f:Y\sr ft(X)$ one has $l(f^*(X))>0$, $ft(f^*X)=Y$ and the square
\begin{eq}
\llabel{2009.10.14.eq1}
\begin{CD}
f^*X @>q(f,X)>> X\\
@Vp_{f^*X}VV @VVp_XV\\
Y @>f>> ft(X)
\end{CD}
\end{eq}
commutes,
\item for $X\in Ob(CC)$ such that $l(X)>0$ one has $id_{ft(X)}^*(X)=X$ and $q(id_{ft(X)},X)=id_X$,
\item for $X\in Ob(CC)$ such that $l(X)>0$, $g:Z\sr Y$ and $f:Y\sr ft(X)$ one has $(gf)^*(X)=g^*(f^*(X))$ and $q(gf,X)=q(g,f^*X)q(f,X)$.
\end{enumerate}
\end{definition}
\begin{remark}\rm
In this definition $pt$ stands for ``point'' as a common notation for a final object of a category. The name ``ft'' stands for ``father'' which is the name given to this map in \cite[Def. 1.1]{Streicher}.
\end{remark}

For $f:Y\sr X$ in $CC$ we let $ft(f):Y\sr ft(X)$ denote the composition $f\circ p_X$.
\begin{definition}
\llabel{2014.07.06.def1}
A C-system is a C0-system together with an operation $f\mapsto s_f$ defined for all $f:Y\sr X$ such that $l(X)>0$ and such that
\begin{enumerate}
\item $s_f:Y\sr (ft(f))^*(X)$,
\item $s_f\circ p_{(ft(f))^*(X)}=Id_Y$,
\item $f=s_f\circ q(ft(f),X)$,
\item if $X=g^*(U)$ where $g:ft(X)\sr ft(U)$ then $s_f=s_{f\circ q(g,U)}$.
\end{enumerate}
\end{definition}
\begin{proposition}
\llabel{2014.07.06.prop1}
Let $CC$ be a C0-system. Then the following are equivalent:
\begin{enumerate}
\item the canonical squares (\ref{2009.10.14.eq1}) of $CC$ are pull-back squares,
\item there is given a structure of a C-system on $CC$.
\end{enumerate}
\end{proposition}
\begin{proof}
Let us show first that if we are given an operation $f\mapsto s_f$ satisfying the conditions of Definition \ref{2014.07.06.def1} then the canonical squares of $CC$ are pull-back squares.

Let $l(X)>0$ and $f: Y\sr ft(X)$. We want to show that for any $Z$ the map 
$$(g:Z\sr f^*(X))\mapsto (ft(g), g\circ q(f,X))$$
is injective and that for any $g_1:Z\sr Y$, $g_2:Z\sr X$ such that $g_1\circ f = ft(g_2)$ there exists a unique $g:Z\sr Y$ such that $ft(g)=g_1$ and $g\circ q(f,X)=g_2$.

Let $g,g':Z\sr f^*(X)$ be such that $ft(g)=ft(g')$ and $g\circ q(f,X)=g'\circ q(f,X)$. Then
$$g=s_{g}\circ q(ft(g),f^*(X))=s_{g\circ q(f,X)}\circ q(ft(g),f^*(X))=$$
$$s_{g'\circ q(f,X)}\circ q(ft(g'),f^*(X))=s_{g'}\circ q(ft(g'),f^*(X))=g'.$$

If we are given $g_1, g_2$ as above let $g=s_{g_2}\circ q(g_1, f^*(X))$. Then:
$$ft(g)=s_{g_2}\circ ft(q(g_1, f^*(X)))=s_{g_2}\circ p_{g_1^*(f_*(X))}\circ g_1=g_1$$
$$g\circ q(f,X)=s_{g_2}\circ q(g_1, f^*(X))\circ q(f,X)=s_{g_2}\circ q(g_1\circ f, X)=s_{g_2}\circ q(ft(g_2),X)=g_2.$$

If on the other hand the canonical squares of $CC$ are pull-back then we can define the operation $s_f$ in the obvious way and moreover such an operation is unique because of the uniqueness part of the definition of pull-back. This implies the assertion of the proposition. 
\end{proof} 
\begin{remark}\rm
As was pointed out by one of the referees,  operation $s_f$ was considered for contextual categories by Cartmell who denoted it by $f\mapsto `f`$, see \cite[2.19]{Cartmell0}.
\end{remark}
\begin{remark}\rm
Let 
$$Ob_n(CC)=\{X\in Ob(CC)\,|\,l(X)=n\}$$
$$Mor_{n,m}(CC)=\{f:Mor(CC)|\partial_0(f)\in Ob_n\,and\,\partial_1(f)\in Ob_m\}.$$
One can reformulate the definitions of C0-systems and C-systems using $Ob_n(CC)$ and $Mor_{n,m}(CC)$ as the underlying sets together with the obvious analogs of maps and conditions of the definition given above. In this reformulation there will be no use of the function $l$ and of the condition $l(X)>0$. 

This shows that C0-systems and C-systems can be considered as models of essentially algebraic theories with sorts $Ob_n$, and $Mor_{n,m}$ and in particular all the results of \cite{Palmgren1} are applicable to C-systems. 
\end{remark}
\begin{remark}\rm
Note that as defined C0-systems and C-systems can not be described, in general, by generators and relations. For example, what is a C0-system generated by $X\in Ob$? There is no such universal object because we do not know what $l(X)$ is. 

This problem is, of course, eliminated by using the definition with two infinite families of sorts $Ob_n$ and $Mor_{n,m}$. 
\end{remark}
\begin{remark}\rm
The notion of a homomorphism of C0-systems and C-systems and the associated definitions of the categories of C0-systems and C-systems are obtained by the specialization of the corresponding general notions for models of essentially algebraic theories. Equivalently homomorphisms are defined as homomorphisms of pre-categories that commute with the length functions and the operations. The category of C-systems is a full subcategory of the category of C0-systems.  Since they are categories of models of essentially algebraic theories they have all limits and colimits. According to the results and observations in \cite{Cartmell0} the category of C-systems is equivalent to a suitably defined category of the GATs (Generalized Algebraic Theories). The category of GATs is studied in \cite{Garner}. 

Presentation of C-systems in terms of GATs uses constructions that are substantially non-finitary - a C-system given by finite sets of generators and relations can rarely be represented by a generalized algebraic theory with finitely many generating objects. 

The C-systems that correspond to finitely presented GATs may play a special role in the theory of C-systems but what such a role might be remains to be discovered.
\end{remark}
\begin{remark}\rm
Note that the additional structure on a pre-category that defines a C0-system is not an additional essentially algebraic structure and can not be made to be such by modification of definitions. Indeed, the pre-category underlying the product of two C0-systems (defined as the categorical product in the category of C0-systems and their homomorphisms) is not the product of the underlying pre-categories but a sub-pre-category in this product which consists of pairs of objects $(X,Y)$ such that $l(X)=l(Y)$. 
\end{remark}

\subsection{The set $\wt{Ob}$ of a C-system.}

For a C-system $CC$ denote by $\wt{Ob}(CC)$ the subset of $Mor(CC)$ which consists of elements $s$ of the form $s:ft(X)\sr X$ where $l(X)>0$ and such that $s\circ p_X=Id_{ft(X)}$. In other words, $\wt{Ob}$ is the set of sections of the canonical projections $p_X$ for $X$ such that $l(X)>0$. 

Note that $f\mapsto s_f$ is an operation from $\{f:Y\sr X|\,l(X)>0\}$ to $\wt{Ob}$. 

For $X\in Ob(CC)$ and $i\ge 0$ such that $l(X)\ge i$ denote by $p_{X,i}$ the composition of the canonical projections $X\sr ft(X)\sr\dots\sr ft^i(X)$ such that $p_{X,0}=Id_X$ and for $l(X)>0$, $p_{X,1}=p_X$.  If $l(X)<i$ we will consider $p_{X,i}$ to be undefined.  {\em All of the considerations involving $p_{X,i}$'s below are modulo the qualification that $p_{X,i}$ is defined, i.e., that $l(X)\ge i$.}

For $X$ such that $l(X)\ge i$ and $f:Y\sr ft^i(X)$ denote by $f^*(X,i)$ the objects and by $q(f,X,i):f^*(X,i)\sr X$ the morphisms defined inductively by the rule 
$$f^*(X,0)=Y\,\,\,\,\,\,\,\,\,q(f,X,0)=f,$$
$$f^*(X,i+1)=q(f,ft(X),i)^*(X)\,\,\,\,\,\,\,\,\,q(f,X,i+1)=q(q(f,ft(X),i), X).$$
If $l(X)<i$, then $q(f,X,i)$ is undefined since $q(-,X)$ is undefined for $X=pt$ and again, as in the case of $p_{X,i}$,  {\em all of the considerations involving $q(f,X,i)$ are modulo the qualification that $l(X)\ge i$}. 

For $i\ge 1$, $(s:ft(X)\sr X)\in \wt{Ob}$ such that $l(X)\ge i$, and $f:Y\sr ft^i(X)$ let 
$$f^*(s,i):f^*(ft(X),i-1)\sr f^*(X,i)$$
be the pull-back of the section $s:ft(X)\sr X$ along the morphism $q(f,ft(X),i-1)$ i.e. the only morphism such that
$$f^*(s,i)\circ p_{f^*(X,i)}=Id_{f^*(ft(X),i-1)}$$
$$f^*(s,i)\circ q(f,X,i)=q(f,ft(X),i-1)\circ s$$
We again use the agreement that always when $f^*(s,i)$ is used the condition $l(X)\ge i$ is part of the assumptions. 

Consider the following operations on the pair of sets $Ob=Ob(CC)$ and $\wt{Ob}=\wt{Ob}(CC)$:
\begin{enumerate}
\item $pt\in Ob$,
\item $ft:Ob\sr Ob$,
\item $\partial:\wt{Ob}\sr Ob$ of the form $(s:ft(X)\sr X)\mapsto X$,
\item $T$ which is defined on pairs $(Y,X)\in Ob\times Ob$ such that $l(Y)>0$ and there exists (a necessarily unique) $l(X)\ge i\ge 1$ with $ft(Y)=ft^i(X)$ and for such pairs $T(Y,X)=p_Y^*( X,i)$,
\item $\wt{T}$ which is defined on pairs $(Y,(r:ft(X)\sr X))\in Ob\times\wt{Ob}$ such that $l(Y)> 0$ and there exists (a necessarily unique) $l(X)\ge i\ge 1$ such that $ft(Y)=ft^i(X)$ and for such pairs $\wt{T}(Y,r)=p_Y^*(r,i)$,
\item $S$ which is defined on pairs $((s:ft(Y)\sr Y),X)\in \wt{Ob}\times Ob$ such that there exists (a necessarily unique) $i\ge 1$ such that $Y=ft^i(X)$ and for such pairs $S(s,X)=s^*(X,i)$,
\item $\wt{S}$ which is defined on pairs $((s:ft(Y)\sr Y),(r:ft(X)\sr X))\in \wt{Ob}\times\wt{Ob}$ such that there exists (a necessarily unique) $i\ge 1$ such that $Y=ft^i(X)$ and for such pairs $\wt{S}(s,r)=s^*(r,i)$,
\item $\delta$ which is defined on elements $X\in Ob$ such that $l(X)>0$ and for such elements $\delta(X)\in \wt{Ob}$ is $s_{Id_X}:X\sr p_X^*(X)$.
\end{enumerate}

\subsection{C-subsystems.} 

A C-subsystem $CC'$ of a C-system $CC$ is a sub-pre-category of the underlying pre-category which is closed, in the obvious sense under the operations which define the C-system on $CC$. 

A C-subsystem is itself a C-system with respect to the induced structure. 
\begin{lemma}
\llabel{2009.10.15.l1}
Let $CC$ be a C-system and $CC'$, $CC''$ be two C-subsystems such that $Ob(CC')=Ob(CC'')$ (as subsets of $Ob(CC)$) and $\wt{Ob}(CC')=\wt{Ob}(CC'')$ (as subsets of $\wt{Ob}(CC)$). Then $CC'=CC''$.
\end{lemma}
\begin{proof}
Let $f:Y\sr X$ be a morphism in $CC'$. We want to show that it belongs to $CC''$. Proceed by induction on $m=l(X)$. For $m=0$ the assertion is obvious. Suppose that $m>0$. Since $CC'$ is a C-subsystem we have a commutative diagram
\begin{eq}
\llabel{2009.11.07.oldeq1}
\begin{CD}
Y\\
@Vs_fVV\\
(f\circ p_X)^*X @>q(f\circ p_X,X)>> X\\
@VVV @VVp_XV\\
Y @>f\circ p_X>> ft(X)
\end{CD}
\end{eq}
in $CC'$ such that $f= s_f\,q(p_{X}f,X)$. By the inductive assumption $f\circ p_X$ is in $CC''$ and since the square is the canonical pull-back square we conclude that $q(p_{X}f,X)$ is in $CC''$. On the other hand $s_f\in CC''$ since $\wt{Ob}(CC')=\wt{Ob}(CC'')$. Therefore $f\in CC''$. 
\end{proof}
\begin{remark}\rm
In Lemma \ref{2009.10.15.l1}, it is sufficient to assume that $\wt{Ob}(CC')=\wt{Ob}(CC'')$. The condition $Ob(CC')=Ob(CC'')$ is then also satisfied. Indeed, let $X\in Ob(CC')$ and $l(X)>0$. Then $p_X^*X$ is the product $X\times_{ft(X)} X$ in $CC$. Consider the diagonal section $\delta_X:X\sr p_X^*X$ of $p_{p_X^*(X)}$. Since $CC'$ is assumed to be a C-subsystem  we conclude that $\delta_X\in \wt{Ob}(CC')=\wt{Ob}(CC'')$ and therefore $X\in Ob(CC'')$. It is however more convenient to think of C-subsystems in terms of subsets of both $Ob$ and $\wt{Ob}$.
\end{remark}
\begin{proposition}
\llabel{2009.10.15.prop2}
A pair $(B,\wt{B})$ where $B\subset Ob(CC)$ and $\wt{B}\subset \wt{Ob}(CC)$ corresponds to a C-subsystem of $CC$ if and only if the following conditions hold:
\begin{enumerate}
\item $pt\in B$,
\item if $X\in B$ then $ft(X)\in B$,
\item if $s\in \wt{B}$ then $\partial(s)\in B$,
\item if $Y\in B$ and $r\in \wt{B}$ then $\wt{T}(Y,r)\in \wt{B}$,
\item if $s\in \wt{B}$ and $r\in \wt{B}$ then $\wt{S}(s,r)\in \wt{B}$,
\item if $X\in B$ then $\delta(X)\in \wt{B}$. 
\end{enumerate}
\end{proposition}
Conditions (4) and (5) are illustrated by the following diagrams:
$$
\begin{CD}
p_{Y}^*(ft(X),i-1) @>q(p_{Y},ft(X),i-1)>> ft(X)\\
@VVq(p_{Y},ft(X),i-1)^*(r)V @VVrV\\
p_{Y}^*(X,i) @>q(p_{Y},X,i)>> X\\
@VVV @VVp_XV\\
p_{Y}^*(ft(X),i-1) @>q(p_{Y},ft(X),i-1)>> ft(X)\\
@VVV @VVV\\
\dots @. \dots\\
@VVV @VVV\\
Y @>p_{Y}>> ft^i(X)
\end{CD}
\,\,\,\,\,\,\,\,\,\,\,\,\,\,\,\,
\begin{CD}
s^*(ft(X),i-1) @>q(s,ft(X),i-1)>> ft(X)\\
@VVq(s,ft(X),i-1)^*(r)V @VVrV\\
s^*(X,i) @>q(s,X,i)>> X\\
@VVV @VVp_XV\\
s^*(ft(X),i-1) @>q(s,ft(X),i-1)>> ft(X)\\
@VVV @VVV\\
\dots @. \dots\\
@VVV @VVV\\
ft^{i+1}(X) @>s>> ft^i(X)
\end{CD}$$
\begin{proof}
The "only if" part of the proposition is straightforward. Let us prove that for any $(B,\wt{B})$ satisfying the conditions of the proposition there exists a C-subsystem $CC'$ of $CC$ such that $B=Ob(CC')$ and $\wt{B}=\wt{Ob}(CC')$. 

Define a candidate subcategory $CC'$ setting $Ob(CC')=B$ and defining the set $Mor(CC'
)$ of morphisms of $CC'$ inductively by the conditions:
\begin{enumerate}
\item $Y\sr pt$ is in $Mor(CC')$ if and only if $Y\in B$,
\item $f:Y\sr X$ is in $Mor(CC')$  if and only if $X\in B$, $ft(f)\in Mor(CC')$ and $s_f\in \wt{B}$.
\end{enumerate}
(Note that for $(f:Y\sr X)\in Mor(CC')$ one has $Y\in B$ since $s_f:Y\sr (ft(f))^*(X)$).

Let us show that if the conditions of the proposition are satisfied then $(Ob(CC'), Mor(CC'))$ form a C-subsystem of $CC$. 

The subset $Ob(CC')$ contains $pt$ and is closed under $ft$ map by the first two conditions. The following lemma shows that $Mor(CC')$ contains identities and the compositions of the canonical projections.
\begin{lemma}
\llabel{2009.10.16.l1}
Under the assumptions of the proposition, if $X\in B$ and $i\ge 0$ then $p_{X,i}:X\sr ft^i(X)$ is in $Mor(CC')$.
\end{lemma}
\begin{proof}
Let $n=l(X)$ and proceed by decreasing induction on $i$ starting with $n$.  The morphism $p_{X,n}$ is of the form $X \sr pt$ and therefore it belongs to $Mor(CC')$ by the first constructor of $Mor(CC')$. By induction it remains to show that if $X\in B$ and $p_{X,i}\in Mor(CC')$ then $p_{X,i-1}\in Mor(CC')$. We have $ft(p_{X,i-1})=p_{X,i}$ and
$$s_{p_{X,i-1}}=(p_{X,i-1},1)^*(\delta(ft^{i-1}(X)))$$
We have $\delta(ft^{i-1}(X))\in \wt{B}$ by conditions (2) and (6).  The pull-back $(p_{X,i-1},1)^*$ can be expressed as the composition of operations $\wt{T}(ft^j(X),-)$, $j=i-1,\dots,1$ and therefore $s_{p_{X,i-1}}$ is in $\wt{B}$ by repeated application of condition (4). 
\end{proof}
\begin{lemma}
\llabel{2009.10.16.l3}
Under the assumptions of the proposition, let $(r:ft(X)\sr X)\in \wt{B}$, $i\ge 1$,  and $(f:Y\sr ft^i(X))\in Mor(CC')$. Then $f^*(r,i):ft(f^*(X,i))\sr f^*(X,i)$ is in $\wt{B}$. 
\end{lemma}
\begin{proof}
Proceed by increasing induction on the length of $ft^i(X)$. Suppose first that $ft^i(X)=pt$. Then $f=p_{Y,n}$ for some $n$ and the statement of the lemma follows from repeated application of condition (4). Suppose that the lemma is proved for all morphisms to objects of length $j-1$ and let the length of $ft^i(X)$ be $j$. Consider the canonical decomposition $f=s_f q_f$. From it we have $f^*(r,i)=s_f^*(q_f^*(r,i),i)$.  Since $q_f$ is the canonical pull-back of $ft(f)$ we further have $q_f^*(r,i)=(ft(f))^*(r,i+1)$ and therefore
$$f^*(r,i)=s_f^*(ft(f)^*(r,i+1),i)$$
By induction $(ft(f))^*(r,i+1)\in \wt{B}$ and therefore $f^*(r,i)\in \wt{B}$ by condition  (5). 
\end{proof}
\begin{lemma}
\llabel{2009.10.16.l4}
Under the assumptions of the proposition, let $g:Z\sr Y$ and $f:Y\sr X$ be in $Mor(CC')$. Then $gf\in Mor(CC')$. 
\end{lemma}
\begin{proof}
If $X=pt$ the the statement is obvious. Assume that it is proved for all $f$ whose codomain is of length $<j$ and let $X$ be of length $j$. We have $ft(gf)=g\,ft(f)$ and therefore $ft(gf)\in Mor(CC')$ by the inductive assumption. It remains to show that  $s_{gf}\in \wt{B}$. We have the following diagram whose squares are canonical pull-back squares
$$
\begin{CD}
X_{gf} @>>> X_f @>>> X\\
@VVV @VVV @VVp_XV\\
Z @>g>> Y @>ft(f)>> ft(X)
\end{CD}
$$
which shows that $s_{gf}=g^*(s_f,1)$. Therefore, $s_{gf}\in Mor(CC')$ by Lemma \ref{2009.10.16.l3}.
\end{proof}
\begin{lemma}
\llabel{2009.10.16.l5}
Under the assumptions of the proposition, let $X\in B$ and let $f:Y\sr ft(X)$ be in $Mor(CC')$, then $f^*(X)\in B$ and $q(f,X)\in Mor(CC')$.
\end{lemma}
\begin{proof}
Consider the diagram
$$
\begin{CD}
f^*(X) @>q(f,X)>> X \\
@Vs_{q(f,X)}VV @VVs_{Id_X}V\\
q(f,X)^*(p_X^*(X)) @>>> p_X^*(X) @>>> X\\
@VVV @VVV @VVV\\
f^*(X) @>q(f,X)>> X @>>> ft(X)\\
@Vp_{f^*(X)}VV  @VVp_XV @.\\
Y @>f>> ft(X)
\end{CD}
$$ 
where the squares are canonical.  By condition (6) we have $s_{Id_X}=\delta(X)\in \wt{B}$. Therefore, by Lemma \ref{2009.10.16.l3}, we have 
$$s_{q(f,X)}=f^*(\delta(X),2)\in \wt{B}.$$
By condition (3),  $\partial(s_{q(f,X)})\in B$ and therefore 
$$f^*(X)=ft(\partial(s_{q(f,X)}))\in B.$$
by condition (2). Together with the previous lemmas this shows that 
$$ft(q(f,X))=p_{f^*(X)}f\in Mor(CC')$$
and therefore $q(f,X)\in Mor(CC')$. 
\end{proof}
\begin{lemma}
\llabel{2009.10.16.l6}
Under the assumptions of Lemma \ref{2009.10.16.l5}, the square
$$
\begin{CD}
f^*(X) @>q(f,X)>> X\\
@Vp_{f^*(X)}VV @VVp_XV\\
Y @>f>> ft(X)
\end{CD}
$$
is a pull-back square in $CC'$.
\end{lemma}
\begin{proof}
We need to show that for a morphism $g:Z\sr f^*(X)$ such that $g p_{f^*(X)}$ and $g q(f,X)$ are in $Mor(CC')$ one has $g\in Mor(CC')$. We have $ft(g)=g p_{f^*(X)}$, therefore by definition of $Mor(CC')$ it remains to check that $s_g\in \wt{B}$. The diagram of canonical pull-back squares 
$$
\begin{CD}
(f^*(X))_g @>>> f^*(X) @>q(f,X)>> X\\
@VVV @VVV @VVV\\
Z @>ft(g)>> Y @>f>> ft(X)
\end{CD}
$$
shows that $s_g=s_{gq(f,X)}$ and therefore $s_g\in Mor(CC')$.
\end{proof}
To finish the proof of the proposition it remains to show that $Ob(CC')=B$ and $\wt{Ob}(CC')=\wt{B}$. The first assertion is tautological. The second one follows immediately from the fact that for $(s:ft(X)\sr X)\in \wt{Ob}(CC)$ one has $ft(s)=Id_{ft(X)}$ and $s_s=s$. 
\end{proof}

\subsection{Regular congruence relations on C-systems}
The following definition of a regular congruence relation is an abstraction to the contextual categories of the structure that arises from the ``definitional'' equalities between types and terms of a type in dependent type theory. This connection is studied further in \cite{Cofamodule}. 
\begin{definition}
\llabel{2014.07.04.def1}
Let $CC$ be a C-system. A regular congruence relation on $CC$ is a pair of equivalence relations $\sim_{Ob}, \sim_{Mor}$ on $Ob(CC)$ and $Mor(CC)$ respectively such that:
\begin{enumerate}
\item $\sim_{Ob}$ and $\sim_{Mor}$ are compatible with $\partial_0,\partial_1,id,ft, (X\mapsto p_X), ((f,g)\mapsto fg), ((X,f)\mapsto f^*(X))$, $(X,f)\mapsto q(f,X)$ and $f\mapsto s_f$,
\item $X\sim_{Ob} Y$ implies $l(X)=l(Y)$,
\item for any $X,F\in Ob(CC)$, $l(X)>0$ such that $ft(X)\sim_{Ob} F$ there exists $X_F$ such that $X\sim_{Ob} X_F$ and $ft(X_F)=F$,
\item for any $f:X\sr Y$ and $X',Y'$ such that $X'\sim_{Ob} X$ and $Y'\sim_{Ob} Y$ there exists $f':X'\sr Y'$ such that $f'\sim_{Mor} f$,
\end{enumerate}
\end{definition}
\begin{lemma}
\llabel{2014.07.08.l1}
If $R=(\sim_{Ob},\sim_{Mor})$ is a regular congruence relation on $CC$ then there exists a unique C-system $CC/R$ on the pair of sets $(Ob(CC)/\sim_{Ob},Mor(CC)/\sim_{Mor})$ such that the obvious function from $CC$ is a homomorphism of C-systems.
\end{lemma}
\begin{proof}
 Since operations such as composition, $(X,f)\mapsto f^*(X)$ and $(X,f)\mapsto q(f,X)$ are not everywhere defined the condition that $\sim_{Ob}$ and $\sim_{Mor}$ are compatible with operations does not imply that the operations can be descended to the quotient sets. However when we add conditions (3) and (4) of Definition \ref{2014.07.04.def1} we see that the functions from the quotients of the domains of definitions of operations to the domains where quotient operations should be defined are surjective and therefore the quotient operations are defined and satisfy all the relations which the original operations satisfied. 
\end{proof}
\begin{lemma}
\llabel{2014.07.06.l2}
Let $R=(\sim_{Ob},\sim_{Mor})$ be a regular congruence relation on $CC$ and let $\sim_{\wt{Ob}}$ be the restriction of $\sim_{Mor}$ to $\wt{Ob}$. Then one has:
$$\wt{Ob}(CC/R)=\wt{Ob}(CC)/\sim_{\wt{Ob}}$$
\end{lemma}
\begin{proof}
It is sufficient to verify that for $X\in Ob(CC)$ and $t: ft(X)\sr X$ such that $l(X)>0$ and $ft(t)\sim_{Mor} Id_{ft(X)}$ there exists $(s: ft(X)\sr X)\in \wt{Ob}(CC)$ such that $t\sim_{Mor} s$.

We have $t=s_t\circ q(ft(t),X)$. Since $ft(t)\sim_{Mor} Id_{ft(X)}$ we have $t\sim_{Mor} s_t$. 
\end{proof}
\begin{proposition}
\llabel{2014.07.08.prop1}
The function which maps a regular congruence relation $(\sim_{Ob},\sim_{Mor})$ to the pair of equivalence relations $(\sim_{Ob},\sim_{\wt{Ob}})$ on $Ob(CC)$ and $\wt{Ob}(CC)$, where $\sim_{\wt{Ob}}$ is obtained by the restriction of $\sim_{Mor}$, is a bijection to the set of pairs of equivalence relations $(\sim,\simeq)$ satisfying the following conditions:
\begin{enumerate}
\item compatibilities with operations $ft$, $\partial$, $T$, $\wt{T}$, $S$, $\wt{S}$ and $\delta$,
\item $X\sim Y$ implies $l(X)=l(Y)$,
\item for any $X,F\in Ob(CC)$, $l(X)>0$ such that $ft(X)\sim F$ there exists $X_F$ such that $X\sim X_F$ and $ft(X_F)=F$,
\item for any $(s:ft(X)\sr X)\in \wt{Ob}$ and $X'\sim X$ there exists $(s': ft(X')\sr X')\in \wt{Ob}$ such that $s'\simeq s$. 
\end{enumerate}
\end{proposition}
\begin{proof}
Let us show first that the pair defined by a regular congruence relation satisfies the conditions (1)-(4). The compatibilities with operations follow from our definitions of these operations in terms of the C-system structure and the assertion of Lemma \ref{2014.07.08.l1} that the projection to the quotient by a regular congruence relation is a homomorphism of C-systems. 

Conditions (2) and (3) follow directly from the definition of a regular congruence relation. Condition (4) follows easily from condition (4) of Definition \ref{2014.07.04.def1} and Lemma \ref{2014.07.06.l2}. 

Let now $(\sim_{Ob},\sim_1)$ and $(\sim_{Ob},\sim_2)$ be two regular congruence relations such that the restrictions of $\sim_1$ and $\sim_2$ to $\wt{Ob}(CC)$ coincide. Let us show that $f\sim_1 f'$ implies that $f\sim_2 f'$. Let $f\sim_1 f'$.   By induction we may assume that $ft(f)\sim_2 ft(f')$. Then $q(ft(f),\partial_1(f))\sim_2 q(ft(f'),\partial_1(f'))$ and $s_f\sim_2 s_{f'}$. Therefore
$$f=s_f\circ q(ft(f),\partial_1(f))\sim_2 s_{f'}\circ q(ft(f'),\partial_1(f'))=f'$$
This proves injectivity.

To prove surjectivity let $(\sim,\simeq)$ be a pair of equivalence relations satisfying conditions (1)-(4). Let us show that it can be extended to a regular congruence relation on $CC$. 

Define $\sim_{Mor}$ on $Mor_{*,m}$ by induction on $m$ as follows. For $m=0$ we say that $(X_1\sr pt)\sim_{Mor} (X_2\sr pt)$ iff $X_1\sim X_2$.

For $(f_1:X_1\sr Y_1)$, $(f_2:X_2\sr Y_2)$ where $l(Y_1)=l(Y_2)=m+1$ we let $f_1\sim_{Mor} f_2$ iff $ft(f_1)\sim_{Mor} ft(f_2)$ and $s_{f_1}\simeq s_{f_2}$. 

Let us show that if $X_1\sim X_2$ and $i\le n=l(X_1)=l(X_2)$ then $p_{X_1,i}\sim_{Mor} p_{X_2,i}$. We show it by decreasing induction $i$. For $i=n$ it immediately follows from our definition. Let $i<n$. By induction we may assume that 
$$ft(p_{X_1,i})=p_{X_1,i+1}\sim_{Mor}p_{X_2,i+1}=ft(p_{X_2,i})$$
On the other hand since $i<l(X)$ one has
$$s_{p_{X,i}}=\wt{T}(X,\wt{T}(ft(X),\dots,\wt{T}(ft^{i-1}(X),\delta(ft^i(X)))\dots ))$$
which implies that $s_{p_{X_1,i}}\simeq s_{p_{X_2,i}}$ and therefore $p_{X_1,i}\sim_{Mor} p_{X_2,i}$.

In particular, if $X_1\sim X_2$ then $Id_{X_1}=p_{X_1,0}\sim_{Mor} p_{X_2,0}=Id_{X_2}$.

This also shows that the restriction of $\sim_{Mor}$ to $\wt{Ob}$ coincides with $\simeq$. Indeed, for $(s:ft(X)\sr X)\in \wt{Ob}$ one has $s_s=s$ and $ft(s)=Id_{ft(X)}$. Therefore
$$(s_1\sim_{Mor} s_2)=(Id_{ft(X_1)}\sim_{Mor} Id_{ft(X_2)})\wedge (s_1\simeq s_2)\Leftrightarrow (s_1\simeq s_2).$$
The rest of the required properties of $\sim_{Mor}$ are verified similarly.
\end{proof}
\begin{remark}\rm
It is straightforward to see that the projection from a C-system on which a regular congruence relation is defined to the C-system that is defined by this congruence relation according to Lemma \ref{2014.07.08.l1} is an epimorphism in the category of C-systems and their homomorphisms. Categorical characterization of such epimorphisms remains at the moment unknown.
\end{remark}

\def\cprime{$'$}

\end{document}